\documentclass[12pt,reqno]{amsart}
\usepackage{fullpage}
\usepackage{amsmath}
\usepackage{mathrsfs,amssymb,graphicx,verbatim,hyperref}
\usepackage{paralist}

\usepackage{color}
\newcommand{\alert}[1]{\textbf{\color{red}
[[[#1]]]}\marginpar{\textbf{\color{red}**}}\typeout{ALERT:
\the\inputlineno: #1}}

\usepackage{ifthen}

\newcommand{\ifsodaelse}[2]{\ifthenelse{\isundefined{\SODAF}}{#2}{#1}}

\ifsodaelse    {\usepackage{ltexpprt}\usepackage{amsmath}}{}
\usepackage{mathrsfs,amssymb,graphicx,verbatim}
\usepackage{paralist}
\renewcommand{\epsilon}{\varepsilon}

\newcommand\remove[1]{}
\renewcommand{\Pr}{\mathbb{P}}
\newcommand{\HH}{\mathbb{H}}
\newcommand{\rnote}[1]{}
\newcommand{\jnote}[1]{}

\newcommand{\1}{\mathbf{1}}
\newcommand{\e}{\varepsilon}
\newcommand{\R}{\mathbb{R}}

\newcommand{\N}{\mathbb{N}}

\newcommand{\vol}{\mathrm{vol}}

\newcommand{\C}{\mathbb{C}}

\DeclareMathOperator{\diam}{diam}

\newtheorem{theorem}{Theorem}[section]
\newtheorem{lemma}[theorem]{Lemma}

\newtheorem{conjecture}{Conjecture}
\newtheorem{question}[conjecture]{Question}
\newcommand{\eqdef}{\stackrel{\mathrm{def}}{=}}

\date{}

\renewcommand{\le}{\leqslant}
\renewcommand{\ge}{\geqslant}
\renewcommand{\leq}{\leqslant}

\renewcommand{\setminus}{\smallsetminus}
\newcommand\Z{{{\mathbb Z}}}

\include{psfig}

\title{Assouad's theorem with dimension independent of the snowflaking}
\thanks{A.~N. was supported in part by NSF grant CCF-0635078, BSF
grant 2006009, and the Packard Foundation. O.~N. was supported in
part by NSF grant CCF-0635078.}
\author{Assaf Naor}
\address{Courant Institute, New York University, New York NY 10012}
\email{naor@cims.nyu.edu}
\author{Ofer Neiman}
\address{Computer Science Department, Princeton University, Princeton NJ 08544-2087}
\email{oneiman@princeton.edu}
\date{}

\begin{document}
\maketitle
\begin{abstract}
It is shown that for every $K>0$ and $\e\in (0,1/2)$ there exist $N=N(K)\in \N$ and $D=D(K,\e)\in (1,\infty)$ with the following properties. For every separable metric space $(X,d)$ with doubling constant at most $K$, the metric space $(X,d^{1-\e})$ admits a bi-Lipschitz embedding into $\R^N$ with distortion at most $D$. The classical Assouad embedding theorem makes the same assertion, but with  $N\to \infty$ as $\e\to 0$.
\end{abstract}


\section{Introduction}\label{sec:intro}

In this paper all metric spaces are assumed to be separable and
contain at least two points. Balls in metric spaces are always
closed balls, i.e.,  for a metric space $(X,d)$, $x\in X$ and $r\ge
0$, we denote $B(x,r)= \{y\in X:\ d(x,y)\le r\}$. A metric space
$(X,d)$ has doubling constant $K\in (1,\infty)$ if every ball in $X$
can be covered by at most $K$ balls of half its radius, i.e., for
every $x\in X$ and $r>0$ there exist $A\subseteq X$ with $|A|\le K$
such that $B(x,r)\subseteq \bigcup_{y\in A} B(y,r/2)$. Note that
since $X$ contains at least two points, necessarily $K\ge 2$. $(X,d)$ is said to be a doubling metric space if it has  doubling constant $K$ for some $K\in (1,\infty)$.

A metric space $(X,d)$ embeds into a normed space $(Y,\|\cdot\|)$
with distortion $D\in [1,\infty]$ if there exists $f:X\to Y$ such
that for all $x,y\in X$ we have $d(x,y)\le \|f(x)-f(y)\|\le
Dd(x,y)$. When $X$ embeds into $Y$ with finite distortion we say
that $X$ admits a bi-Lipschitz embedding into $Y$. The infimum over
those $D\ge 1$ for which $X$ embeds into $Y$ is denoted $c_Y(X)$.
When $Y=\ell_2$ is infinite dimensional Hilbert space, we write
$c_Y(X)=c_2(X)$; this parameter is known in the literature as the {
Euclidean distortion of $X$}. In what follows, when we refer to the
space $\R^N$ we always assume that it is equipped with the standard
Euclidean metric. A standard argument (see,
e.g.,~\cite[Lem.~4.9]{Ass83}) shows
 that if $Y$ is either $\ell_2$ or
$\R^N$, we have
\begin{equation}\label{eq:reduce to finite subsets}
c_Y(X)=\sup\{c_Y(Z):\ Z\subseteq X\ \wedge\  |Z|<\infty\}.
\end{equation}

If $(X,d)$ is a metric space and $\alpha\in (0,1]$ then $(X,d^\alpha)$ is also a metric space, known as the $\alpha$-snowflake of $X$.

A major open problem of embedding theory is  the  bi-Lipschitz
embeddability problem in $\R^N$. This problem asks for an intrinsic
characterization of those separable metric spaces $(X,d)$ that admit
a bi-Lipschitz embedding into  $\R^N$ for some $N\in \N$. For a
discussion of this important question, see for example the works of
Semmes~\cite{Sem99}, Lang-Plaut~\cite{LP01} and
Heinonen~\cite{Hei03}. An obvious restriction on a metric space
$(X,d)$ that admits a bi-Lipschitz embedding into $\R^N$ is that it
must be {doubling}. In this context, Assouad discovered
in~\cite{Ass83} the following fundamental embedding theorem (see
also Heinonen's book~\cite{Hei01} for a nice exposition of Assouad's
theorem).

\begin{theorem}[Assouad's embedding theorem]
For every $\e\in (0,1)$ and $K>0$, there exist $N=N(K,\e)\in \N$ and
$D=D(K,\e)\in (1,\infty)$ such that for every separable metric space
$(X,d)$ with doubling constant $K$, the metric space
$\left(X,d^{1-\e}\right)$ admits a bi-Lipschitz embedding into
$\R^N$ with distortion at most $D$.
\end{theorem}

Assouad's theorem falls short of solving the bi-Lipschitz
embeddability problem in $\R^N$, since it only achieves an embedding
of the snowflaked metric space $(X,d^{1-\e})$. Nevertheless, as
$\e\to 0$ this metric space becomes closer and closer to the
original metric space $(X,d)$. It is therefore of interest to
investigate the behavior of $N(K,\e)$ and $D(K,\e)$ as $\e\to 0$. It
turns out that necessarily $\lim_{\e\to 0}D(K,\e)=\infty$, due to
the existence of doubling metric spaces that do not admit a
bi-Lipschitz embedding into $\R^N$. The first known such example  is
the Heisenberg group, equipped with the Carnot-Carath\'eodory
metric: Semmes observed in~\cite{Sem96} that its bi-Lipschitz
nonembeddability into $\R^N$ is a consequence of Pansu's
differentiability theorem~\cite{Pan89}. Additional examples of
non-Euclidean doubling spaces were found by Laakso~\cite{Laa02} and
Bourdon-Pajot~\cite{BP99}; see the work of Cheeger~\cite{Che99} for
a unified treatment of these results.

It seems to be inherent to Assouad's embedding method that also
$\lim_{\e\to 0}N(K,\e)= \infty$. Note that if $\e\in (0,1/2)$ then
the metric space $(X,d^{1-\e})$ has doubling constant $K^2$, so there
is no obvious obstruction to $(X,d^{1-\e})$ admitting a bi-Lipschitz
embedding into $\R^N$ for some $N$ that is independent of $\e \in
(0,1/2)$. The issue that in Assouad's theorem $N$ depends on $\e$
and is very large as $\e\to 0$ was noted by many authors; this is
mentioned, for example, in the works of David-Toro~\cite{DT99} and
Semmes~\cite{Sem99} (where much more refined bounds on $N$ are
obtained under additional assumptions). Assouad himself noticed this
issue in~\cite{Ass83}, where he showed that $N$ can be taken to be
independent of $\e\in (0,1/2)$ when $X=\R$ (more generally, Assouad
deals in~\cite{Ass83} with $X=[0,1]^k$). The case of the ``helix
snowflakes" $(\R,|x-y|^{1-\e})$ was studied by Kahane~\cite{Kah81}
and Talagrand~\cite{Tal92}, who investigated the interplay between
the dimension $N$ and the distortion $D$ (Kahane studied only the
case $\e=\frac12$, and obtained sharp results. Talagrand's work
applies to all $\e\in (0,1)$, but is not sharp).

Here we show that in Assouad's theorem one can take $N$ to depend
only on the doubling constant $K$, but not on $\e\in (0,1/2)$.

\begin{theorem}\label{thm:main}
For every $K>0$ there exists $N=N(K)\in \N$, and for every $\e\in
(0,1/2)$ and $K>0$ there exists $D=D(K,\e)\in (1,\infty)$, such that
for every separable metric space $(X,d)$ with doubling constant $K$,
the metric space $\left(X,d^{1-\e}\right)$ admits a bi-Lipschitz
embedding into $\R^N$ with distortion at most $D$.
\end{theorem}
Our argument yields the bounds $N(K)\lesssim \log K$ and
$D(K,\e)\lesssim \left(\frac{\log K}{\e}\right)^2$. More generally,
for every $\delta\in (0,1]$ our argument yields the bounds
\begin{equation}\label{eq:our bounds}
N(K)\lesssim \frac{\log K}{\delta}\quad\mathrm{and}\quad
D(K,\e)\lesssim  \left(\frac{\log K}{\e}\right)^{1+\delta}.
\end{equation}
Here and in what follows, the symbols $\lesssim, \gtrsim$ indicate
the corresponding inequalities up to an absolute multiplicative
factor.

In the rest of this introduction we will describe some additional
results and question related to the bi-Lipschitz embeddability
problem in $\R^N$.

\subsection{The Lang-Plaut problem and snowflakes of the Heisenberg
group} Despite major efforts by many mathematicians, the
bi-Lipschitz embeddability problem in $\R^N$ remains wide open. A
variety of sufficient intrinsic conditions on a metric space $(X,d)$
are known which ensure that it admits a bi-Lipschitz embedding in
some Euclidean space $\R^N$, but these conditions are far from
necessary.

A necessary condition for a metric space $(X,d)$ to admit a
bi-Lipschitz embedding into some $\R^N$ (in addition to being
doubling) is that it admits a bi-Lipschitz embedding into $\ell_2$,
i.e., its Euclidean distortion satisfies $c_2(X)<\infty$. All the
known examples of doubling metric spaces that do not admit a
bi-Lipschitz embedding into any $\R^N$ actually do not admit a
bi-Lipschitz embedding into infinite dimensional Hilbert space as
well. This led Lang and Plaut~\cite[Question~2.4]{LP01} to ask the
following question.

\begin{question}[Lang-Plaut problem]\label{Q:LP}
Is it necessary and sufficient for a metric space $(X,d)$ to admit a
bi-Lipschitz embedding into some $\R^N$ that it is doubling and it
admits a bi-Lipschitz embedding into Hilbert space? Equivalently,
does every doubling subset of Hilbert space admit a bi-Lipschitz
embedding into some $\R^N$?
\end{question}
By a simple argument (presented in Section~\ref{sec:heisenberg}), the Lang-Plaut problem
can be restated quantitatively as follows. Is it true that for every
$K>0$ there is $N=N(K)\in \N$ and $D=D(K)\in (1,\infty)$ such that
if $X\subseteq \ell_2$ has doubling constant $K$ then
$c_{\R^N}(X)\le D$?

One might argue whether or not a positive answer to the Lang-Plaut
problem would resolve the bi-Lipschitz embedding problem into
$\R^N$, since it is not obvious that the condition that $X$ admits a
bi-Lipschitz embedding into Hilbert space can be restated in terms of
the intrinsic geometry of $X$. But, it is possible to characterize bi-Lipschitz embeddability into $\ell_2$ in terms of a family of distance inequalities, i.e., intrinsically, without using the word ``embedding". Indeed, as shown by
Linial-London-Rabinovich~\cite{LLR95} (extending the corresponding
classical result of Schoenberg~\cite{Sch38} in the isometric
category), $c_2(X)\le D$ if and only if for all $n\in
\N$, $x_1,\ldots,x_n\in X$ and every $n\times n$ symmetric positive
semidefinite matrix $Q=(q_{ij})$, all of whose rows sum to $0$, the
following inequality holds true:
\begin{equation}\label{eq:LLR}
\sum_{i=1}^n\sum_{j=1}^n \max\{q_{ij},0\} d(x_i,x_j)^2\le
D^2\sum_{i=1}^n \sum_{j=1}^n \max\{-q_{ij},0\} d(x_i,x_j)^2.
\end{equation}
Hence, a positive answer to the Lang-Plaut question would yield a
characterization of bi-Lipschitz embeddability into some $\R^N$ in
terms of the doubling condition, and the family of distance
inequalities~\eqref{eq:LLR}. We believe that this would yield a
satisfactory answer to the bi-Lipschitz embeddability problem in
$\R^N$, though there does not seem to be evidence supporting a
positive answer to the the Lang-Plaut question.

A potential source of doubling subsets of Hilbert space that might
yield a counter-example to the Lang-Plaut problem is Assouad's
theorem itself. When allowing embeddings into infinite dimensional
Hilbert space rather than into $\R^N$, the asymptotics in terms of
$\e$ of $D(K,\e)$ in Assouad's theorem are known~\cite{LNM05} (see
also~\cite{NS10}). Specifically, if $(X,d)$ has doubling constant
$K$ then $c_2(X)\le C(K)/\sqrt{\e}$ for some $C(K)\in (0,\infty)$. This
dependence on $\e$ is sharp up to the value of $C(K)$, as shown
in~\cite[Remark~5.4]{LNM05}.

If $(X,d)$ has doubling  constant $K$ then the space $(X,d^{1-\e})$
has doubling constant bounded uniformly in $\e\in (0,1/2)$, but, in
its $C(K)/\sqrt{\e}$-distortion embedding into $\ell_2$ it might
have an image that is not a doubling subset of $\ell_2$, with
doubling constant independent of $\e$, due to the large distortion.
We therefore ask the following question:
\begin{question}\label{Q:doubling image}
Is it true that for every $K\in (1,\infty)$ there exist $a(K),b(K)\in
(0,\infty)$ with the following property. If $(X,d)$ has doubling
constant $K$ and $\e\in (0,1/2)$ then there exists $f:X\to \ell_2$ such that
$a(K)\sqrt{\e}d(x,y)^{1-\e}\le \|f(x)-f(y)\|\le d(x,y)^{1-\e}$ for all $x,y\in X$,
and $f(X)\subseteq \ell_2$ has doubling constant $b(K)$.
\end{question}
Observe that due to Theorem~\ref{thm:main}, with the explicit bounds stated
in~\eqref{eq:our bounds}, if we replaced in Question~\ref{Q:doubling
image} the term $\sqrt{\e}$ by $\e^{1+\delta}$ for any $\delta\in(0,1]$, then the answer would
be positive, and even the image of the embedding would be finite
dimensional with dimension depending only on $K$ and $\delta$.

In spite of the fact that we don't know the answer to
Question~\ref{Q:doubling image}, we do know that the answer is
positive for the Heisenberg group. For $n\in \N$, the $n$'th
Heisenberg group $\HH_n$ is $\C^n\times \R$, equipped with the
following group product:
\begin{multline*}
(w,s)\cdot (z,t)=\left(w+z,s+t+2\sum_{j=1}^n
\Im\left(w_j\overline{z_j}\right)\right)\\ \forall
w=(w_1,\ldots,w_n),z=(z_1,\ldots,z_n)\in \C^n,\ \forall s,t\in \R.
\end{multline*}
Thus $(0,0)$ is the identity of $\HH_n$ and for $(z,t)\in \HH_n$ we
have $(z,t)^{-1}=(-z,-t)$.

The Koranyi norm on $\HH_n$ is defined for $(z,t)\in \HH_n$ by
$N_0(z,t)= \sqrt[4]{|z|^4+t^2}$,
where $|z|^2=\sum_{j=1}^n|z_j|^2$. For $g,h\in \HH_n$ we have
$N_0(gh^{-1})\le N_0(g)+N_0(h)$ (see~\cite{KR85,Cyg81}). Thus
$d_{N_0}(g,h)= N_0(h^{-1}g)$ is a left-invariant metric on $\HH_n$.
One can check that the Lebesgue measure is a Haar measure of
$\HH_n$, and that $(\HH_n, d_{N_0})$ has doubling constant
$e^{O(n)}$.

In Section~\ref{sec:heisenberg} we observe that a result of~\cite{LN06} implies the following statement.

\begin{theorem}\label{thm:heis}
For every $\e\in (0,1/2)$ and for every $n\in \N$, there exists $f_\e:\HH_n\to \ell_2$ satisfying  $\sqrt{\e} d_{N_0}(x,y)^{1-\e}\le \|f_\e(x)-f_\e(y)\|\le d_{N_0}(x,y)^{1-\e}$ for all $x,y\in \HH_n$, and such that $f_\e(\HH_n) $ is a doubling subset of $\ell_2$, with doubling constant $e^{O(n)}$.
\end{theorem}

We also show in Section~\ref{sec:heisenberg} that
Theorem~\ref{thm:heis} is sharp, even without the requirement that
the image of $\HH_n$ is doubling with constant independent of $\e$:
\begin{equation}\label{eq:lower c_2 heis}
c_2\left(\HH_n,d_{N_0}^{1-\e}\right)\gtrsim \frac{1}{\sqrt{\e}}\quad \forall\ \e\in (0,1/2).
\end{equation}

This raises the following question:

\begin{question}\label{eq:no lowdim}
Is it true that for every fixed $N\in \N$ we have $\lim_{\e\to 0} c_{\R^N}\left(\HH_1,d_{N_0}^{1-\e}\right)\sqrt{\e}=\infty$?
\end{question}
A positive answer to Question~\ref{eq:no lowdim} would imply a
negative answer to the Lang-Plaut problem, since otherwise there
would be $N\in\N$ and $D\in (1,\infty)$ satisfying
$c_{\R^N}(f_\e(\HH_1))\le D$ for all $\e\in (0,1/2)$, where $f_\e$
is the Euclidean embedding of $\left(\HH_1,d_{N_0}^{1-\e}\right)$
from Theorem~\ref{thm:heis}. This would yield the bound
$c_{\R^N}\left(\HH_1,d_{N_0}^{1-\e}\right)\sqrt{\e}\le D$.

\subsection{Previous work and an overview of the proof of Theorem~\ref{thm:main}} The
classical proof of Assouad's theorem~\cite{Ass83,Hei01} yields the
dimension bound $N(K,\e)\le c(K)/\e^{O(1)}$. In~\cite{GKL03}
Gupta-Krauthgamer-Lee announced a similar bound on $N(K,\e)$ with a
much better dependence of $c(K)$ on $K$, yet the same bound in terms
of $\e$ (the proof of this assertion of~\cite{GKL03} hasn't appeared
since the 2003 announcement, and in particular the dependence on
$\e$ was not stated there explicitly, but it seems to us that the
proof technique suggested in~\cite{GKL03} would lead to this bound).
A similar bound follows from the work of Har-Peled and
Mendel~\cite{HM06}, who studied in addition embeddings into
$\ell_\infty^N$, yielding a $1+\delta$ distortion result. The best
previously known bound is due to Abraham-Bartal-Neiman~\cite{ABN},
who proved that $N(K,\e)\le c(K)\log(1/\e)$. In the context of the
Lang-Plaut problem, Gottlieb-Krauthgamer~\cite{GK11}, and
Bartal-Recht-Schulman~\cite{BRS11}, proved that if $X\subseteq
\ell_2$ has doubling constant $K$ then for all $\delta\in (0,1)$ the
$(1-\e)$-snowflake of $X$ embeds with distortion $1+\delta$ into
$\R^{c(K,\delta)/\e^{O(1)}}$; the main point in these works,
however, is to obtain a $1+\delta$ distortion embedding, which is
impossible in the context of general doubling metric spaces that are
not necessarily isometric to a subset of $\ell_2$.

Our proof of Theorem~\ref{thm:main} builds heavily on the method of
Abraham-Bartal-Neiman~\cite{ABN}. In essence, our proof should be
viewed as an optimization of the argument of~\cite{ABN} which uses
degrees of freedom that were available in the construction
of~\cite{ABN} but were not previously exploited. This requires
subtle changes in the proof of~\cite{ABN}, and in particular we were
surprised that such changes can lead to a complete removal of the
dependence on $\e$ of the dimension $N$  in Assouad's theorem.
Though somewhat delicate, these changes are of a technical nature,
and the key conceptual ideas can all be found in~\cite{ABN}.

 The proof of Theorem~\ref{thm:main} is based on a construction of a distribution over random embeddings, arising
 from a certain family of random multi-scale partitions of the metric space $(X,d)$. At every possible distance scale we provide a mapping to $\R$ which is essentially the truncated distance to the ``boundary" of the random partition. We then combine all the possible scales into a single embedding into $\R$, using an idea of Assouad~\cite{Ass83} which multiplies every scale by an appropriate factor that enables us to control the total expansion over all scales. The lower bound on the distance of the image of every pair of points in $X$ will come from a single critical scale. Instead of showing sufficient contribution for \emph{every} pair, we first focus on certain nets of the space at appropriate scales,
 showing that this suffices to prove the desired lower bound on all
pairs. The bulk of the proof consists of arguing that not only the
net pairs will have sufficient contribution, but that this will
happen with high probability (depending on $\e$), and with very few
dependencies on other net points. To show this we use, as
in~\cite{ABN}, a localization property of the ``padding event" of
the random partitions: this event is stochastically independent of
the ``far away" structure of the partition. The ball expected to be
padded is very small (which causes additional distortion), but on
the other hand the padding probability is high. The fact that there
is non-constant distortion in the lower bound forces us to define
the original distance scales  to be also be a function of $\e$.
Finally, to argue that the desired lower bound happens for all pairs
with positive probability, even though the number of dimensions at
our disposal is small, we use the Lov\'asz Local Lemma.

\subsection*{\bf Acknowledgements} We are grateful to Tim Austin and
Bruce Kleiner for helpful discussions.

\section{Preliminaries}


Due to~\eqref{eq:reduce to finite subsets} it suffices to prove
Theorem~\ref{thm:main} when $X$ is finite, provided that the
resulting distortion $D(K,\e)$ and dimension $N(K)$ do not depend on
$|X|$. We will therefore assume from now on that $X$ is finite. This
assumption is actually not necessary for our argument, but it serves
the role of allowing us to ignore measurability issues that might
arise in the random partitioning arguments.

For a partition $P$ of $X$ and $x\in X$ let $P(x)\in P$ be the set in $P$ to which $x$ belongs. For $s>0$ the partition $P$ is called $s$-bounded if the diameter of $P(x)$ is at most $s$ for all $x\in X$.

There is a canonical way to obtain partitions from balls. Given $x_1,\ldots,x_n\in X$ and $r_1,\ldots,r_n\in (0,\infty)$, define a partition $P^{x_1,\ldots,x_n}_{r_1,\ldots,r_n}$ of $\bigcup_{j=1}^n B(x_j,r_j)$ by
\begin{equation}\label{eq:def part}
P^{x_1,\ldots,x_n}_{r_1,\ldots,r_n}\eqdef\{B(x_1,r_1)\}\bigcup\left\{B(x_j,r_j)\setminus \bigcup_{i=1}^{j-1}B(x_i,r_i)\right\}_{j=2}^n\setminus\{\emptyset\}.
\end{equation}
In particular, given $s>0$ the partition $P^{x_1,\ldots,x_n}_{r_1,\ldots,r_n}$ is an $s$-bounded partition of $X$ whenever $\{x_1,\ldots,x_n\}$ is an $s/4$-net of $X$ and $r_1,\ldots,r_n\in [s/4,s/2]$.

As in~\cite{ABN}, we will use random partitions of the form
$P^{x_1,\ldots,x_n}_{r_1,\ldots,r_n}$, where the radii
$r_1,\ldots,r_n$ are appropriately chosen random variables. We
present the proofs of the necessary properties of these partitions
below, even though they follow from~\cite{ABN}. We do so since the
argument of~\cite{ABN} is carried out in much greater generality
because in~\cite{ABN} these methods are used for other purposes for
which more general constructions are needed. Our argument below is
simpler than the proof in~\cite{ABN} both because it deals with the
special case that we need, but also because the proof here is
different from~\cite{ABN} (relying, of course, on the same ideas).

\begin{lemma}\label{lem:dist-property}
Fix $x \in X$. For $s>0$ and $K>1$ let $R$ be a random variable with the following density 
\begin{equation}\label{eq:density}
\phi_s(r)\eqdef\frac{16K^8\log K}{s(K^4-1)}K^{-16r/s}\1_{[s/4,s/2]}(r).
\end{equation}
Then for every $\beta>0$ and every $y\in X$ we have
\begin{multline}\label{eq:double intersecton}
\Pr\left [ B(x,R)\cap B(y,\beta s)\notin\big\{\emptyset, B(y,\beta
s)\big\}\right]\\ \leq \left(1-K^{-32\beta}\right) \left(\Pr
\left[B(y, \beta s) \cap B(x,R)\neq \emptyset \right] +
\frac{1}{K^4-1}\right).
\end{multline}
\end{lemma}
\begin{proof}
Fix $x,y \in X$ and define
$$
a\eqdef\min_{z\in B(y,\beta s)}d(x,z)\quad \mathrm{and}\quad
b\eqdef\max_{z\in B(y,\beta  s)}d(x,z).
$$
By the triangle inequality,
\begin{equation}\label{eq:b-a}
b-a \leq 2\beta s.
\end{equation}
Hence,
\begin{multline}\label{eq:ab}
\Pr\left [ B(x,R)\cap B(y,\beta s)\notin\big\{\emptyset, B(y,\beta
s)\big\}\right]=\int_{\max\{a,s/4\}}^b\phi_s(r)dr\\\le\frac{K^8}{K^4-1}\left(K^{-16\max\{a,s/4\}/s}-K^{-16b/s}\right)\stackrel{\eqref{eq:b-a}}{\le}
\frac{K^8}{K^4-1} K^{-16\max\{a,s/4\}/s}\left(1-K^{-32\beta}\right).
\end{multline}
Similarly,
\begin{equation}\label{eq:a}
\Pr \left[B(y, \beta s) \cap B(x,R)\neq \emptyset \right] =
\int_{\max\{a,s/4\}}^{s/2} \phi_s(r)dr  =
\frac{K^8}{K^4-1}\left(K^{-16\max\{a,s/4\}/s} -K^{-8}\right).
\end{equation}
The desired inequality~\eqref{eq:double intersecton} now follows from~\eqref{eq:ab} and~\eqref{eq:a}.
\end{proof}

\begin{lemma}\label{lem:pad}
Fix $s>0$, $K\ge 2$. Assume that $(X,d)$ has doubling constant at
most $K$. Let $\{x_1,\ldots,x_n\}\subseteq X$ be an $s/4$-net of
$X$, and let $R_1,\ldots,R_n$ be i.i.d. random variables whose
distribution is given by~\eqref{eq:density}. Then for every $y\in X$
and $\beta\in (0,1/40)$ we have,
$$
\Pr\left[B(y,\beta s)\subseteq P^{x_1,\ldots,x_n}_{R_1,\ldots,R_n}(y)\right]\ge K^{-64\beta}.
$$
\end{lemma}

\begin{proof}
For every $j\in \{1,\ldots,n\}$ consider the following event:
$$
A_j\eqdef \left(\bigcap_{i=1}^{j-1}\left\{B(x_i,R_i)\cap B(y,\beta s)=\emptyset\right\}\right)\bigcap
\big\{B(x_j,R_j)\cap B(y,\beta s)\notin\left\{\emptyset, B(y,\beta s)\right\}\big\}.
$$
For $A_j$ to occur we need in particular to have $B(x_j,R_j)\cap B(y,\beta s)\neq \emptyset$. Since $R_j,\beta s\le s/2$, this implies that $j\in J_y$, where  $$J_y\eqdef \left\{j\in \{1,\ldots,n\}:\ x_j\in B\left(y,s\right)\right\}.$$ We can cover $B(y,s)$ by at most $K^3$ balls of radius $s/8$. Since $\{x_1,\ldots,x_n\}$ is an $s/4$-net, each of these balls can contain at most one of the $x_i$. Thus
\begin{equation}\label{eq:J-y}
|J_y|\le K^3.
\end{equation}
We must have $B(x_j,R_j)\cap B(y,\beta s)\neq \emptyset $ for some $j\in J_y$, and therefore using the independence of $R_1,\ldots R_n$ we see that,
\begin{multline}\label{eq:sum to 1}
1=\sum_{j\in J_y}\Pr\left[B(x_j,R_j)\cap B(y,\beta s)\neq \emptyset\ \wedge \ B(x_i,R_i)\cap B(y,\beta s)=\emptyset\ \forall i<j\right]\\=
\sum_{j\in J_y}\left(\prod_{i=1}^{j-1} \Pr\left[B(x_i,R_i)\cap B(y,\beta s)=\emptyset\right]\right)
\Pr\left[ B(x_j,R_j)\cap B(y,\beta s)\neq \emptyset\right].
\end{multline}

Now, by the definition of the partition $P^{x_1,\ldots,x_n}_{R_1,\ldots,R_n}$ we have
$$
\left\{B(y,\beta s)\not\subseteq P^{x_1,\ldots,x_n}_{R_1,\ldots,R_n}(x)\right\}=\bigcup_{j\in J_y} A_j.
$$
Thus, using the independence of $R_1,\ldots R_n$ once more,
\begin{eqnarray}\label{eq:with K}
&&\!\!\!\!\!\!\!\!\!\!\!\!\!\!\!\!\!\!\nonumber1-\Pr\left[B(y,\beta s)\subseteq P^{x_1,\ldots,x_n}_{R_1,\ldots,R_n}(x)\right]\le \sum_{j\in J_y} \Pr\left[A_j\right]\\&=&
\sum_{j\in J_y}\left(\prod_{i=1}^{j-1} \Pr\left[B(x_i,R_i)\cap B(y,\beta s)=\emptyset\right]\right)\nonumber
\Pr\big[B(x_j,R_j)\cap B(y,\beta s)\notin\left\{\emptyset, B(y,\beta s)\right\}\big]\\\nonumber
&\stackrel{\eqref{eq:double intersecton}\wedge\eqref{eq:sum to 1}}{\le}& \left(1-K^{-32\beta}\right)+\left(1-K^{-32\beta}\right)\frac{|J_y|}{K^4-1}\\
&\stackrel{\eqref{eq:J-y}}{\le} & \left(1-K^{-32\beta}\right)\left(1+\frac{K^3}{K^4-1}\right)\nonumber\\
&\le& 1-K^{-64\beta}\label{eq:numeric},
\end{eqnarray}
Where in~\eqref{eq:numeric} we used the fact that since $K\ge 2$ and $\beta<1/40$, we have $\frac{K^3}{K^4-1}\le K^{-32\beta}$. Indeed, this is equivalent to $32\beta\le \frac{\log(K-K^{-3})}{\log K}$. But, the function $K\mapsto \frac{\log(K-K^{-3})}{\log K}$ is increasing on $(1,\infty)$, since its derivative is $\frac{4}{(K^5-K)\log K}+\frac{\log(K^4/(K^4-1))}{K(\log K)^2}\ge 0$. Thus it suffices to check that $32\beta\le \frac{\log(2-2^{-3})}{\log 2}$, which is true since $\beta<1/40$.
\end{proof}

\section{The random embedding}
Fix $\e\in (0,1/2)$, $\theta\in (0,1)$ and $K\ge 2$. Write $K=e^\kappa$, and define
\begin{equation}\label{eq:def N}
N\eqdef \left\lceil \frac{c\kappa}{\theta}\right\rceil=\left\lceil
\frac{c\log K}{\theta}\right\rceil,
 \end{equation}
 where $c>0$ is a universal constant that will be determined later. It will also be convenient to write
 \begin{equation}\label{eq:def tau}
 \tau\eqdef \frac{\e^\theta}{32\kappa^\theta}.
 \end{equation}

 Let $(X,d)$ be a finite metric space whose doubling constant is at most $K$. By normalization assume that $\diam(X)=1$. For every $i\in \N$ let $\{x_1^i,\ldots,x_{n_i}^i\}$ be a $\frac14\tau^{i/(1-\e)}$-net of $X$.
For every $i,k\in \N$ and $j\in \{1,\ldots,n_i\}$ let $R_{ij}^k$ be
a random variable whose density is $\phi_{s}$, as given
in~\eqref{eq:density}, with $s=\tau^{i/(1-\e)}$. We will also use
random variables $\{U_i^k(C):\ i,k\in\N,\ C\subseteq X\}$, each of
which is uniformly distributed on the interval $[0,1]$ (thus for
each $i,k\in \N$ we have $2^{|X|}$ such random variables).
Throughout the argument below it is assumed that the random
variables
\begin{equation}\label{eq:all variables}
\Big\{R_{ij}^k:\ i,k\in \N,\ j\in\{1,\ldots,n_i\}\Big\}\bigcup
\Big\{U_i^k(C):\ i,k\in\N,\ C\subseteq X\Big\}
 \end{equation}
 are mutually independent and defined on some probability space $(\Omega, \Pr)$.

We will now consider the random partitions
\begin{equation}\label{eq:def P_i^k}
P_i^k\eqdef P_{R_{i1}^k,\ldots,R_{in_i}^k}^{x_1^i,\ldots,x_{n_i}^i},
\end{equation}
where $P_{R_{i1}^k,\ldots,R_{in_i}^k}^{x_1,\ldots,x_{n_i}}$ is defined as in~\eqref{eq:def part}. For $i\in \N$ and $k\in \{1,\ldots,N\}$ define a random mapping $f_i^k:X\to \R$ by
\begin{equation}\label{def:basic f}
f_i^k(x)\eqdef U_i^k\left(P_i^k(x)\right)\cdot \min\left\{
\tau^i,64\kappa \tau^{-\frac{i\e}{1-\e}-1}d\left(x,X\setminus
P_i^k(x)\right)\right\}.
\end{equation}
Finally, we define a random embedding $F:X\to \R^N$ as follows:
\begin{equation}\label{eq:def F}
F(x)=\left(\frac{\sum_{i=1}^\infty f_i^1(x)}{\sqrt{N}},\ldots,\frac{\sum_{i=1}^\infty f_i^N(x)}{\sqrt{N}}\right)\in \R^N.
\end{equation}
Note that by the definition of $f_i^k$, the sums appearing in~\eqref{eq:def F} converge geometrically.

Although $F$ is random, it satisfies the desired $(1-\e)$-H\"older condition deterministically. The randomness will enter when we prove that with positive probability $\|F(x)-F(y)\|_2$ satisfies the desired lower bound for all $x,y\in X$.

\begin{lemma}\label{lem: holder}
For every $x,y\in X$ we have
\begin{multline}\label{eq:Holder}
\|F(x)-F(y)\|_2\le\max_{k\in
\{1,\ldots,N\}}\sum_{i=1}^{\infty}\left|f_i^k(x)-f_i^k(y)\right|
\lesssim
\frac{\kappa^{(1+\theta)(1-\e)}}{\e^{1+\theta}}d(x,y)^{1-\e}\\\lesssim
\left(\frac{\log K}{\e}\right)^{1+\theta}d(x,y)^{1-\e}.
\end{multline}
\end{lemma}

\begin{proof}
We first claim that for all $i\in \N$ and $k\in \{1,\ldots,N\}$ we have
\begin{equation}\label{eq:upper min}
\left|f_i^k(x)-f_i^k(y)\right|\le \min\left\{
\tau^i,64\kappa \tau^{-\frac{i\e}{1-\e}-1}d(x,y)\right\}.
\end{equation}
To verify~\eqref{eq:upper min} we may assume without loss of generality that $f_i^k(x)>f_i^k(y)$.
If $P_i^k(x)\neq P_i^k(y)$ then,
\begin{equation*}
f_i^k(x)-f_i^k(y)\le f_i^k(x)\le \min\left\{
\tau^i,64\kappa \tau^{-\frac{i\e}{1-\e}-1}d\left(x,X\setminus P_i^k(x)\right)\right\},
\end{equation*}
which is trivially bounded from above by the right hand side
of~\eqref{eq:upper min} since $y\in X\setminus P_i^k(x)$. If
$P_i^k(x)=P_i^k(y)=C$, then it cannot be the case that
$f_i^k(y)=U_i^k\left(C\right)\tau^{i}$, since otherwise $f_i^k(x)\le
f_i^k(y)$, contrary to our assumption. We therefore necessarily have
\begin{multline}\label{eq:conjunction}
f_i^k(x)-f_i^k(y)=f_i^k(x)-64U_i^k(C)\kappa \tau^{-\frac{i\e}{1-\e}-1}d(y,X\setminus C)\\
\le 64U_i^k(C)\kappa \tau^{-\frac{i\e}{1-\e}-1}\left(d(x,X\setminus
C)-d(y,X\setminus C)\right)\le 64\kappa
\tau^{-\frac{i\e}{1-\e}-1}d(x,y).
\end{multline}
Since, by the definition~\eqref{def:basic f}, $f_i^k(x),f_i^k(y)\in [0,\tau^i]$, we also have $f_i^k(x)-f_i^k(y)\le \tau^i$. This, in conjunction with~\eqref{eq:conjunction}, concludes the proof of~\eqref{eq:upper min}.

The first inequality of~\eqref{eq:Holder} is an immediate
consequence of the definition~\eqref{eq:def F}. As $K=e^\kappa$, the
third inequality in~\eqref{eq:Holder} is a trivial overestimate.
Note also that since for all $z\in X$ we have $\|F(z)\|_2\le
\sum_{i=1}^\infty \tau^i\lesssim \tau$, the bound in the second
inequality of~\eqref{eq:Holder} holds true if
$d(x,y)>\tau^{1+\frac{1}{1-\e}}/(64\kappa)$. We may therefore assume
that $d(x,y)\le \tau^{1+\frac{1}{1-\e}}/(64\kappa)$. Let $m\in \N$
be the integer satisfying
\begin{equation}\label{eq:m}
\frac{\tau}{64\kappa}\cdot\tau^{\frac{m+1}{1-\e}}< d(x,y)\le\frac{\tau}{64\kappa}\cdot\tau^{\frac{m}{1-\e}}.
\end{equation}
Then
\begin{multline}\label{eq:to estimate two sums}
\|F(x)-F(y)\|_2\le \sqrt{\frac{1}{N}\sum_{k=1}^N\left(\sum_{i=1}^{\infty}\left|f_i^k(x)-f_i^k(y)\right|\right)^2}\le \max_{k\in \{1,\ldots,N\}}\sum_{i=1}^{\infty}\left|f_i^k(x)-f_i^k(y)\right|\\\stackrel{\eqref{eq:upper min}}{\lesssim}
\frac{\kappa}{\tau}d(x,y)\sum_{i=1}^{m}\tau^{-\frac{i\e}{1-\e}}+\sum_{i=m+1}^\infty \tau^i\stackrel{\eqref{eq:def tau}}{\lesssim} \frac{\kappa^{1+\theta}}{\e^\theta}d(x,y)\sum_{i=1}^{m}\tau^{-\frac{i\e}{1-\e}}+\sum_{i=m+1}^\infty \tau^i.
\end{multline}
We estimate the two sums in~\eqref{eq:to estimate two sums} separately (recalling that $0<\e, \tau<1/2$):
\begin{equation}\label{eq:first sum}
\sum_{i=1}^{m}\tau^{-\frac{i\e}{1-\e}}=
\frac{\tau^{-\frac{m\e}{1-\e}}-1}{1-\tau^{\frac{\e}{1-\e}}}\lesssim \frac{1}{\e}\tau^{-\frac{m\e}{1-\e}}\stackrel{\eqref{eq:m}}{\lesssim} \frac{\tau^\e}{\e\kappa^\e d(x,y)^\e}\stackrel{\eqref{eq:def tau}}{\lesssim}
\frac{1}{\e\kappa^{(1+\theta)\e}d(x,y)^\e}.
\end{equation}
Similarly,
\begin{equation}\label{eq:second sum}
\sum_{i=m+1}^\infty \tau^i=
\frac{\tau^{m+1}}{1-\tau} \stackrel{\eqref{eq:m}}{\lesssim} \frac{\kappa^{1-\e}}{\tau^{1-\e}} d(x,y)^{1-\e}\stackrel{\eqref{eq:def tau}}{\lesssim}\frac{\kappa^{(1+\theta)(1-\e)}}{\e^\theta}d(x,y)^{1-\e}.
\end{equation}
The desired bound~\eqref{eq:Holder} now follows from
substituting~\eqref{eq:first sum} and~\eqref{eq:second sum}
into~\eqref{eq:to estimate two sums}.
\end{proof}


\subsection{The H\"older lower bound holds with positive probability}

  For every $i\in \N$ write
   \begin{equation}\label{eq:def delta i}
  \delta_i\eqdef \tau^{\frac{i+2}{1-\e}}\left(\frac{4\e}{c^*\kappa}\right)^{\frac{1}{1-\e}},
  \end{equation}
  where $c^*$ is the implied universal constant in the final inequality of~\eqref{eq:Holder}. Let $\mathscr N_i$ be a $\delta_i$-net of $X$.

  Consider the following set:
  \begin{equation}\label{eq:def M}
  M\eqdef \left\{(i,u,v)\in \N\times \mathscr N_i\times \mathscr N_i:\ \tau^{\frac{i}{1-\e}}< d(u,v)\le 3\tau^{\frac{i-1}{1-\e}}\right\}.
  \end{equation}
  For every $(i,u,v)\in M$ define $G(i,u,v)\subseteq \{1,\ldots,N\}$ as follows.
  \begin{equation}\label{eq:def G}
  G(i,u,v)\eqdef \left\{k\in \{1,\ldots,N\}:\ \left|\sum_{j=1}^\infty f_j^k(u)-\sum_{j=1}^\infty f_j^k(v)\right|\ge \frac{\tau^{i+1}}{2}\right\}.
  \end{equation}

  For every $(i,u,v)\in M$ let $E(i,u,v)\subseteq \Omega$ be the following event:
  \begin{equation}\label{eq:def E(i,u,v)}
  E(i,u,v)\eqdef \left\{|G(i,u,v)|\ge \frac{N}{2}\right\},
  \end{equation}
  and consider the event $E\subseteq \Omega$ given by:
  \begin{equation}\label{eq:def E}
  E\eqdef \bigcap_{(i,u,v)\in M}E(i,u,v).
  \end{equation}
  The relevance of the event $E$ is explained in the following lemma.
  \begin{lemma}\label{lem:on E}
  If the event $E$ occurs then for all $x,y\in X$ we have
  $$
  \|F(x)-F(y)\|_2\gtrsim \left(\frac{\e}{\log K}\right)^{2\theta}d(x,y)^{1-\e}.
  $$
  \end{lemma}
  \begin{proof}
  Let $i$ be the integer such that
  \begin{equation}\label{eq:i choice}
\tau^{\frac{i}{1-\e}}< d(x,y)\le \tau^{\frac{i-1}{1-\e}}
  \end{equation}
  Since $\mathscr N_i$ is a $\delta_i$-net, where $\delta_i$ is given in~\eqref{eq:def delta i}, there exist $u,v\in \mathscr N_i$ such that
  \begin{equation}\label{eq:u,v choice}
  \max\{d(u,x),d(v,y)\}\le \tau^{\frac{i+2}{1-\e}}\left(\frac{4\e}{c^*\kappa}\right)^{\frac{1}{1-\e}}.
  \end{equation}
Assume that $k\in G(i,u,v)$. By Lemma~\ref{lem: holder} we have
\begin{multline}\label{eq:use holder for net move}
\max\left\{\left|\sum_{j=1}^\infty f_j^k(u)-\sum_{j=1}^\infty f_j^k(x)\right|,\left|\sum_{j=1}^\infty f_j^k(v)-\sum_{j=1}^\infty f_j^k(y)\right|\right\}\\
\stackrel{\eqref{eq:Holder}\wedge\eqref{eq:u,v choice}}{\le} c^*\left(\frac{\kappa}{\e}\right)^{1+\theta} \frac{4\tau^{i+2}\e}{c^*\kappa}\stackrel{\eqref{eq:def tau}}{=} \frac{\tau^{i+1}}{8}.
\end{multline}
Since $k\in G(i,u,v)$ it follows that
\begin{equation}\label{eq:subtract}
\left|\sum_{j=1}^\infty f_j^k(x)-\sum_{j=1}^\infty f_j^k(y)\right|\stackrel{\eqref{eq:def G}\wedge\eqref{eq:use holder for net move}}{\ge}
 \frac{\tau^{i+1}}{2}-2\cdot \frac{\tau^{i+1}}{8}=\frac{\tau^{i+1}}{4}.
\end{equation}
  Since we are  assuming that the event $E$ occurs, the lower bound~\eqref{eq:subtract} holds for at least $N/2$ values of $k\in \{1,\ldots,N\}$. Thus, by the definition of $F$,
  \begin{equation*}
  \|F(x)-F(y)\|_2\gtrsim \tau^{i+1}\stackrel{\eqref{eq:def tau}\wedge\eqref{eq:i choice}}{\gtrsim} \left(\frac{\e}{\kappa}\right)^{2\theta}d(x,y)^{1-\e}.\qedhere
  \end{equation*}
  \end{proof}

  Due to Lemma~\ref{lem: holder} and Lemma~\ref{lem:on E}, Theorem~\ref{thm:main} will be
  proven (with the bounds claimed in~\eqref{eq:our bounds}, with
  $\delta=3\theta$), once we establish the following lemma.

  \begin{lemma}\label{lem:E positive}
 We have  $\Pr[E]>0$, provided $c$ in~\eqref{eq:def N} is a large enough universal constant. \end{lemma}

  The key tool used in the proof of Lemma~\ref{lem:E positive} is the Lov\'asz Local
 Lemma~\cite{EL75}. The variant of this lemma that is stated below is not the
same as the classical formulation of the Lov\'asz Local Lemma, but
it is a consequence of it, as explained in~\cite{ABN}, where a more
general statement is needed. For more information on the Lov\'asz
Local Lemma and some of its striking applications, see for example
the survey of Alon~\cite{Alo89}.

  \begin{lemma}[Lov\'asz Local Lemma]\label{lem:local} Fix $q\in (0,1)$ and $d\in \N$.
Let $\mathcal{A}_1,\mathcal{A}_2,\dots\mathcal{A}_n$ be measurable sets in
some probability space $(\Omega,\Pr)$. Let $G=(V,E_G)$ be a graph on the vertex set $V=\{\mathcal{A}_1,\mathcal{A}_2,\dots\mathcal{A}_n\}$ with maximal degree $d$. Let $\rho:\{\mathcal{A}_1,\mathcal{A}_2,\dots\mathcal{A}_n\}\to \N$ be a mapping that satisfies the condition
$$
\{\mathcal{A}_i,\mathcal{A}_j\}\in E_G \implies \rho(\mathcal{A}_i)= \rho(\mathcal{A}_j).
$$
Assume that for any
$i\in\{1,\dots,n\}$ we have
\[
\Pr\left[\bigcap_{j\in
Q}(\Omega\setminus \mathcal{A}_j)\cap \mathcal{A}_i\right] \le q \Pr\left[\bigcap_{j\in
Q}(\Omega\setminus \mathcal{A}_j)\right]
\]
for all
$$
Q\subseteq\left\{j\in\{1,\ldots,n\}:\ \{\mathcal{A}_i,\mathcal{A}_j\}\notin E_G\ \wedge\
\rho(\mathcal{A}_i)\ge \rho(\mathcal{A}_j)\right\}.
$$
Assume also that
\begin{equation}\label{LLL condition}
eq(d+1)\le 1.
\end{equation}
Then
\[
\Pr\left[\bigcap_{i=1}^n(\Omega\setminus\mathcal{A}_i)\right]>0
\]
\end{lemma}

To use Lemma~\ref{lem:local} we proceed as follows. For  $(i,u,v)\in
M$ consider the following random subset of $\{1,\ldots,k\}$:
\begin{equation}\label{eq:def L}
L(i,u,v)\eqdef \left\{k\in \{1,\ldots,N\}:\ \left|\sum_{j=1}^i f_j^k(u)-\sum_{j=1}^i f_j^k(v)\right|\ge 2 \tau^{i+1}\right\}.
\end{equation}
For $(i,u,v)\in M$ and $k\in \{1,\ldots,N\}$ define the following event:
\begin{equation}\label{eq:def S}
S(i,u,v,k)\eqdef\{k\in L(i,u,v)\}.
\end{equation}
Finally, we also define the following event for all $(i,u,v)\in M$:
\begin{equation}\label{eq:def T}
T(i,u,v)\eqdef \left\{|L(i,u,v)|\ge \frac{N}{2}\right\}.
\end{equation}

  \begin{lemma}\label{lem:inclusion}
  For all $(i,u,v)\in M$ we have $T(i,u,v)\subseteq E(i,u,v)$.
  \end{lemma}

  \begin{proof}
  Using~\eqref{eq:upper min} we see that for all $k\in \{1,\ldots,N\}$,
  \begin{equation}\label{eq:3/2}
  \left|\sum_{j=i+1}^\infty f_j^k(u)-\sum_{j=i+1}^\infty f_j^k(v)\right|\le \sum_{j=i+1}^\infty \tau^j\le \frac32\tau^{i+1}.
  \end{equation}
  Hence, if $k\in L(i,u,v)$ then
  \begin{equation*}
  \left|\sum_{j=1}^\infty f_j^k(u)-\sum_{j=1}^\infty f_j^k(v)\right|\ge \left|\sum_{j=1}^i f_j^k(u)-\sum_{j=1}^i f_j^k(v)\right|-
  \left|\sum_{j=i+1}^\infty f_j^k(u)-\sum_{j=i+1}^\infty f_j^k(v)\right|
  \stackrel{\eqref{eq:def L}\wedge\eqref{eq:3/2}}{\ge} \frac{\tau^{i+1}}{2}.
  \end{equation*}
  This means that $L(i,u,v)\subseteq G(i,u,v)$, and hence $|L(i,u,v)|\ge \frac{N}{2}\implies |G(i,u,v)|\ge \frac{N}{2}$.
  \end{proof}

Before proceeding with the proof of Lemma~\ref{lem:E positive}, it is
beneficial for us to introduce some notation related to the random
partitions that are used in the definition of the embedding $F$. For
$i\in \N$ and $k\in \{1,\ldots,N\}$ the partition $P_i^k$ was
defined in~\eqref{eq:def P_i^k}, where $\{x_1^i,\ldots,x_{n_i}^i\}$
is a fixed $\frac14 \tau^{i/(1-\e)}$-net of $X$, and
$R^k_{i1},\ldots,R^k_{in_i}$ are i.i.d. random variables whose
density is $\phi_{s}$ as given in~\eqref{eq:density}, with
$s=\tau^{i/(1-\e)}$. For every $y\in X$ define
\begin{equation}\label{def:J_y}
J(i,y)\eqdef \left\{j\in \{1,\ldots,n_i\}:\ d\left(y,x_j^i\right)\le
2\tau^{\frac{i}{1-\e}}\right\}.
\end{equation}
We will consider the following random variable
\begin{equation}\label{eq:def j(i,y)}
j(i,k,y)\eqdef \min\left\{j\in J(i,y): y\in B\left(x^i_j,R^k_{ij}\right)\right\}.
\end{equation}
To see that $j(i,k,y)$ is well-defined, note that since $\{x_1^i,\ldots,x_{n_i}^i\}$ is a  $\frac14 \tau^{i/(1-\e)}$-net of $X$ and $R^k_{i1},\ldots,R^k_{in_i}\ge \frac14 \tau^{i/(1-\e)}$, there must be some $j\in \{1,\ldots,N\}$ for which $y\in B\left(x^i_j,R^k_{ij}\right)$, and since $R^k_{ij}\le \frac12 \tau^{i/(1-\e)}$ necessarily $j\in J(i,y)$.

>From the definition~\eqref{eq:def part} we see that
\begin{equation}\label{eq:compute part}
P_i^k(y)= B\left(x^i_{j(i,k,y)},R^k_{ij(i,k,y)}\right)\setminus \bigcup_{\ell=1}^{j(i,k,y)-1} B\left(x^i_\ell,R^k_{i\ell}\right).
\end{equation}
But note that if there exists $z\in
B\left(x^i_\ell,R^k_{i\ell}\right)\cap
B\left(x^i_{j(i,k,y)},R^k_{ij(i,k,y)}\right)$  then
$$
d\left(x^i_\ell,y\right)\le
d\left(x^i_\ell,z\right)+d\left(z,x^i_{j(i,k,y)}\right)+
d\left(x^i_{j(i,k,y)},y\right)\le  R^k_{i\ell}+ 2R^k_{ij(i,k,y)}\le
2\tau^{\frac{i}{1-\e}},
$$
implying that $\ell\in J(i,y)$ . It follows from this that~\eqref{eq:compute part} can be rewritten as follows:
\begin{equation}\label{eq:compute part 2}
P_i^k(y)= B\left(x^i_{j(i,k,y)},R^k_{ij(i,k,y)}\right)\setminus \bigcup_{\ell\in J(i,y)\cap \{1,\ldots,j(i,k,y)-1\}} B\left(x^i_\ell,R^k_{i\ell}\right).
\end{equation}

  To continue with our plan to use Lemma~\ref{lem:local}, we  define a graph $H=(V,E_H)$, where
  $V\eqdef \{T(i,u,v):\ (i,u,v)\in M\}$, and a mapping $\rho:V\to \N$, as follows.
  \begin{equation}\label{eq:def edges}
  \left\{T(i,u,v),T(i',u',v')\right\}\in E_H\iff i=i'\ \wedge\ d\left(\{u,v\},\{u',v'\}\right)\le 4\tau^{\frac{i}{1-\e}},
  \end{equation}
  \begin{equation}\label{eq:def rho}
  \rho(T(i,u,v))=i.
  \end{equation}

  \begin{lemma}\label{lem:degree bound}
  The maximal degree of $H$ is at most $K^{c^{**}(\log\log K+\log(1/\e))}$,
  where $c^{**}\in (0,\infty)$ is a universal constant.
  \end{lemma}

  \begin{proof}
  Given $(i,u,v)\in M$, we need to bound the number of $(i,u',v')\in M$ satisfying $d\left(\{u,v\},\{u',v'\}\right)\le 4\tau^{i/(1-\e)}$. We may assume that $d\left(u,u'\right)\le 4\tau^{i/(1-\e)}$. Recall that from the definition of M in~\eqref{eq:def M} we know that $d(u,v),d(u',v')\le 3\tau^{(i-1)/(1-\e)}$. Hence the points $v,u',v'$ are all in the ball $B$ of radius $r= 4\tau^{(i-1)/(1-\e)}$ centered at $u$, implying that the number of $(i,u',v')$ as above is at most $|B\cap \mathscr N_i|^2$. Since $(X,d)$ is $K$-doubling, $B$ can be covered by at most $K^{1+\log_2(2r/\delta_i)}$ balls of radius $\delta_i/2$, each of which contains at most one point from the $\delta_i$-net $\mathscr N_i$ (recall~\eqref{eq:def delta i} for the definition of $\delta_i$). Hence, the maximal degree of $H$ is at most
  \begin{equation*}
  K^{4+2\log_2(r/\delta_i)}=K^{O(\log\log K+\log(1/\e))}.\qedhere
  \end{equation*}
  \end{proof}

\begin{lemma}\label{lem:condisional bound}
For every $(i,u,v)\in M$ and for every \begin{equation}\label{eq:our
Q} Q\subseteq \Big\{(i',u',v')\in M:\ i\ge i'\ \wedge\
\left\{T(i,u,v),T(i',u',v')\right\}\notin E_H\Big\},
\end{equation}
we have
\begin{equation}\label{eq:for lovasz}
\Pr\left[\bigcap_{(i',u',v')\in Q} T(i',v',u')\setminus
T(i,u,v)\right]\le \left(\frac{\e}{\log K}\right)^{\theta
N/2}\Pr\left[\bigcap_{(i',u',v')\in Q} T(i',v',u')\right].
\end{equation}
\end{lemma}

\begin{proof} Denote
$$
W\eqdef \bigcap_{(i',u',v')\in Q} T(i',v',u').
$$
Consider the following subsets $\mathscr X, \mathscr Y$ of the
random variables given in~\eqref{eq:all variables}:
\begin{multline*}
\mathscr{X}\eqdef \Big\{R_{i'j}^k:\ i'\in \{1,\ldots,i-1\},\
j\in\{1,\ldots,n_{i'}\},\ k\in \{1,\ldots,N\}\Big\}\\\bigcup
\Big\{U_{i'}^k(C):\ i'\in\{1,\ldots,i-1\},\ k\in \{1,\ldots,N\},\
C\subseteq X\Big\},
\end{multline*}
\begin{multline*}
\mathscr{Y}\eqdef \Big\{R_{ij}^k:\ j\in \{1,\ldots,n_i\}\setminus
J(i,u),\ k\in \{1,\ldots,N\}\Big\}\\\bigcup \Big\{U_{i}^k(C):\ k\in
\{1,\ldots,N\},\
 C\subseteq X\setminus B\left(u,2\tau^{i/(1-\e)}\right)\Big\}.
\end{multline*}

The event $W$ depends only on the variables
$\mathscr{X}\cup\mathscr{Y}$. Indeed, for $(i',u',v')\in Q$ with
$i'<i$  it follows from the definitions~\eqref{eq:def
T},\eqref{eq:def L},\eqref{def:basic f} that the event $T(i',u',v')$
depends only on the variables $\mathscr{X}$. If $(i,u',v')\in Q$ and
$\left\{T(i,u,v),T(i,u',v')\right\}\notin E_H$ then it follows
from~\eqref{eq:def edges} that $d(u',u),d(v',u)>4\tau^{i/(1-\e)}$.
Since the diameter of $P_i^k(u'),P_i^k(v')$ is at most
$\tau^{i/(1-\e)}$, it follows that $P_i^k(u'),P_i^k(v')\subseteq
X\setminus B\left(u,2\tau^{i/(1-\e)}\right)$, and hence
$U_i^k(P_i^k(u')),U_i^k(P_i^k(v'))\in \mathscr{Y}$. Similarly, we
know that $J(i,u)\cap J(i,u')=\emptyset$ and $J(i,u)\cap
J(i,v')=\emptyset$, and hence from the identity~\eqref{eq:compute
part 2} we know that $P_i^k(u'),P_i^k(v')$ depend only on the
variables $\mathscr{Y}$. These observations, combined with the
definition~\eqref{def:basic f}, imply that $f_i^k(u'),f_i^k(v')$
depend only on the variables $\mathscr{Y}$, and from the
definitions~\eqref{eq:def T},\eqref{eq:def L} we conclude that the
event $T(i,u',v')$ depends only on the variables
$\mathscr{X}\cup\mathscr{Y}$, as required.

Recalling the definitions~\eqref{eq:def L}, \eqref{eq:def S},
\eqref{eq:def T}, it follows from the above argument that
\begin{equation}\label{eq:conditioned}
\Pr\left[W\cap T(i,u,v)\right]=\int_W \Pr\left[\left.\sum_{k=1}^N
\1_{S(i,u,v,k)}\ge
\frac{N}{2}\right|\mathscr{X}\cup\mathscr{Y}\right]d\Pr.
\end{equation}
To estimate the right hand side of~\eqref{eq:conditioned}, for each
$k\in \{1,\ldots,n\}$ consider the event
\begin{equation}\label{eq:def Zk}
Z_k\eqdef S(i,u,v,k)\cap \left\{B\left(u,
\frac{\tau}{64\kappa}\tau^{\frac{i}{1-\e}}\right)\subseteq
P_i^k(u)\right\}.
\end{equation}
From~\eqref{eq:conditioned} we then see that
\begin{equation}\label{eq:conditioned lower bound}
\Pr\left[W\cap T(i,u,v)\right]\ge\int_W \Pr\left[\left.\sum_{k=1}^N
\1_{Z_k}\ge \frac{N}{2}\right|\mathscr{X}\cup\mathscr{Y}\right]d\Pr.
\end{equation}

An application of Lemma~\ref{lem:pad} with
$\beta=\frac{\tau}{64\kappa}$ yields the estimate
\begin{equation}\label{eq:use pad lemma}
\Pr\left[B\left(u,
\frac{\tau}{64\kappa}\tau^{\frac{i}{1-\e}}\right)\subseteq
P_i^k(u)\right]\ge K^{-\tau/\kappa}=e^{-\tau}.
\end{equation}
Moreover, it follows from the definition~\eqref{def:basic f} that if
$B\left(u,
\frac{\tau}{64\kappa}\tau^{\frac{i}{1-\e}}\right)\subseteq P_i^k(u)$
then we have $f_i^k(u)=U_i^k\left(P_i^k(u)\right)\tau^i$. Hence,
recalling the definition~\eqref{eq:def S}, we see that
\begin{multline}\label{eq:determine f}
Z_k=\left\{B\left(u,
\frac{\tau}{64\kappa}\tau^{\frac{i}{1-\e}}\right)\subseteq
P_i^k(u)\right\}\\\bigcap\left\{\left|\sum_{j=1}^{i-1}\left(f_j^k(u)-f_j^k(v)\right)+
U_i^k\left(P_i^k(u)\right)\tau^i-f_i^k(v)\right|\ge
2\tau^{i+1}\right\}.
\end{multline}
>From the identity~\eqref{eq:compute part 2} we see that the event
$\left\{B\left(u,
\frac{\tau}{64\kappa}\tau^{\frac{i}{1-\e}}\right)\subseteq
P_i^k(u)\right\}$ is independent of $\mathscr{X}\cup\mathscr{Y}$.
Thus, denoting $a\eqdef
\tau^{-i}\sum_{j=1}^{i-1}\left(f_j^k(u)-f_j^k(v)\right)-\tau^{-i}f_i^k(v)$,
we have
\begin{eqnarray}
p&\eqdef&\label{eq:def p}\Pr\left[Z_k
\Big|\mathscr{X}\cup\mathscr{Y}
\right]\\\nonumber&\stackrel{\eqref{eq:determine
f}}{=}&\Pr\left[B\left(u,
\frac{\tau}{64\kappa}\tau^{\frac{i}{1-\e}}\right)\subseteq
P_i^k(u)\right]\cdot\Pr\left[ U_i^k\left(P_i^k(u)\right)\notin
\left(a-2\tau, a+2\tau\right)\Big| \mathscr{X}\cup\mathscr{Y}\right]
\\&\stackrel{\eqref{eq:use pad lemma}}{\ge}&
e^{-\tau}(1-4\tau)\label{eq:uniform}\\&\ge& 1-5\tau\label{eq:5tau},
\end{eqnarray}
where in~\eqref{eq:uniform} we used the fact that since $d(u,v)>
\tau^{i/(1-\e)}$ (by the definition~\eqref{eq:def M} of $M$), and
$\diam (P_i^k(u))\le \tau^{i/(1-\e)}$, we have $P_i^k(u)\neq
P_i^k(v)$, and therefore the random variable $U_i^k(P_i^k(u))$,
which is uniformly distributed on $[0,1]$, is independent of $a$ and
$\mathscr{X}\cup\mathscr{Y}$.

Since after fixing the values of $\mathscr{X}\cup\mathscr{Y}$, the
events $Z_1,\ldots,Z_N$  are independent, the Chernoff bound
(see~\cite[Thm.~A.1.12]{AS08}) implies that
\begin{multline}\label{eq:chernoff}
\Pr\left[\left.\sum_{k=1}^N \1_{Z_k}\ge
\frac{N}{2}\right|\mathscr{X}\cup\mathscr{Y}\right]=1-\Pr\left[\left.\sum_{k=1}^N
\1_{\Omega\setminus Z_k}>
\frac{N}{2}\right|\mathscr{X}\cup\mathscr{Y}\right]\\\stackrel{\eqref{eq:def
p}}{\ge} 1-
\left(e^{p-\frac12}\sqrt{2(1-p)}\right)^N\stackrel{\eqref{eq:5tau}}{\ge}
1-(10e\tau)^{N/2}.
\end{multline}
Substituting~\eqref{eq:chernoff} into~\eqref{eq:conditioned lower
bound} shows that
$$
\Pr\left[W\cap T(i,u,v)\right]\ge
\left(1-(30\tau)^{N/2}\right)\Pr[W]\stackrel{\eqref{eq:def
tau}}{\ge} \left(1-\left(\frac{\e}{\log K}\right)^{\theta
N/2}\right)\Pr[W],
$$
which is the same statement as~\eqref{eq:for lovasz}.
\end{proof}

\begin{proof}[Proof of Lemma~\ref{lem:E positive}]
By Lemma~\ref{lem:inclusion} we have
$$
\Pr[E]\ge \Pr\left[\bigcap_{(i,u,v)\in M} T(i,u,v)\right].
$$
Hence, due to Lemma~\ref{lem:degree bound} and
Lemma~\ref{lem:condisional bound}, Lemma~\ref{lem:E positive} will
follow from Lemma~\ref{lem:local} if the condition corresponding
to~\eqref{LLL condition} holds, i.e.,
$$
e\left(\frac{\e}{\log K}\right)^{\theta N/2}\left(K^{O(\log\log
K+\log(1/\e))}+1\right)\le 1.
$$
This holds true provided the constant $c$ in the
definition~\eqref{eq:def N} of $N$ is large enough.
\end{proof}

\section{Snowflakes of the Heisenberg group}\label{sec:heisenberg}

As promised in the introduction, we first argue that a positive
answer to the qualitative version of the Lang-Plaut question, as
appearing in Question~\ref{Q:LP}, implies its quantitative variant,
i.e., that for every $K>0$ there is $N=N(K)\in \N$ and $D=D(K)\in
(1,\infty)$ such that if $X\subseteq \ell_2$ has doubling constant
$K$ then $c_{\R^N}(X)\le D$. Indeed, if not then there would be some
$K>0$ and a sequence $\{X_n\}_{n=1}^\infty$  of subsets of $\ell_2$ with doubling constant $K$ and
satisfying $c_{\R^n}(X_n)> n$. By~\eqref{eq:reduce to finite
subsets} there are finite subsets $F_n\subseteq X_n$ with
$c_{\R^n}(F_n)> n$, and by translation and rescaling we may assume
that $0\in F_n$ and that $F_n$ is contained in the ball of $\ell_2$
centered at $0$ of radius 1. Let $Y\subseteq \ell_2\times \R$ be
given by $Y=\bigcup_{n=1}^\infty F_n\times \{4^n\}$. One checks that
$Y$ has doubling constant $O(K)$, and clearly all the $F_n$ embed
into $Y$ isometrically. By the assumed positive answer to the
Lang-Plaut problem it follows that $c_{\R^N}(F_n)\le D$ for some
$N\in \N$ and $D\in (1,\infty)$, a contradicting the fact that
$c_{\R^N}(F_n)>n$ for $n\ge N$.


\begin{proof}[Proof of Theorem~\ref{thm:heis}]
For $\theta>0$ define $\delta_\theta:\HH_n\to \HH_n$ by
$\delta_\theta(z,t)=(\theta z,\theta^2 t)$. Note that for every
measurable $A\subseteq \HH_n=\C^n\times \R$ we have
$\vol(\delta_\theta(A))=\theta^{2n+2}\vol(A)$. For $p\in [1,2)$ and
$(z,t)\in \HH_n$ define
$$
M_p(z,t)\eqdef
\sqrt[4]{|z|^4+t^2}\left(\cos\left(\frac{p}{2}\arccos\left(\frac{|z|^2}{\sqrt{|z|^4+t^2}}\right)\right)\right)^{1/p}.
$$

It was shown in~\cite{LN06} that $M_p(xy^{-1})\le M_p(x)+M_p(y)$ for
all $x,y\in \HH_n$. Therefore $d_{M_p}(x,y)=M_p(y^{-1}x)$ is a
left-invariant metric on $\HH_n$. It was also shown in~\cite{LN06}
that $\sqrt{1-\frac{p}{2}}N_0(x)\le M_p(x)\le N_0(x)$ for all $x\in
\HH_n$, and there exists $f:\HH_n\to \ell_2$ satisfying
$\|f(x)-f(y)\|=d_{M_p}(x,y)^{p/2}$ for all $x,y\in \HH_n$. Setting
$p=2(1-\e)$, we see that for all distinct $x,y\in \HH_n$ we have
$$
\frac{\|f(x)-f(y)\|}{{d_{N_0}(x,y)^{1-\e}}}=\left(\frac{d_{M_p}(x,y)}{d_{N_0}(x,y)}\right)^{p/2}\in
\left[\e^{(1-\e)/2},1\right]\subseteq [\sqrt{\e},1].
$$

For $x\in \HH_n$ and $r>0$ denote  $B_p(x,r)=\{y\in \HH_n:\
d_{M_p}(x,y)^{p/2}\le r\}$. Note that
$B_p(0,r)=\delta_{2^{2/p}}(B_p(0,r/2))$, since for every $\theta>0$
and $x,y\in \HH_n$ we have
$d_{M_p}\left(\delta_\theta(x),\delta_\theta(y)\right)=\theta
d_{M_p}(x,y)$. Hence, by left-invariance of $d_{M_p}^{p/2}$ and the
Lebesgue measure $\vol(\cdot)$, for all $x\in \HH_n$ and $r>0$ we
have $\vol (B_p(x,r))=2^{4(n+1)/p}\vol(B_p(x,r/2))$. This implies
that
$\left(\HH_n,d_{M_p}^{p/2}\right)=\left(\HH_n,d_{M_p}^{1-\e}\right)$,
and hence also its isometric copy $f(\HH_n)\subseteq \ell_2$, has
doubling constant $2^{8(n+1)/p}\le 2^{16(n+1)}$.
\end{proof}

We end with the proof of the distortion lower bound~\eqref{eq:lower
c_2 heis}. Assume that $f:\HH_1\to \ell_2$ satisfies
\begin{equation}\label{eq:assumption f}
d_{N_0}(x,y)^{1-\e}\le\|f(x)-f(y)\|\le Dd_{N_0}(x,y)^{1-\e}\quad
\forall \  x,y\in\HH_1.
 \end{equation}
Our goal is to prove that $D\gtrsim 1/\sqrt{\e}$. Denote $a=(1,0)\in
\HH_1$, $b=(i,0)\in \HH_1$ and $c=aba^{-1}b^{-1}=(0,-4)$. Writing
 $$
 B_m=\left\{(u+iv,t)\in \HH_1:\ u,v,t\in \Z\ \wedge \ N_0(u+iv,t)\le m\right\},
 $$
it follows from~\cite{ANT10} that there exists a universal constant $C>0$ such that for all $m\in \N$ we have
\begin{equation}\label{eq:ANT}
\sum_{x\in B_m}\sum_{k=1}^{m^2}\frac{\left\|f(xc^k)-f(x)\right\|^2}{k^2}\lesssim \sum_{x\in B_{Cm}}\left(\|f(xa)-f(x)\|^2+\|f(xb)-f(x)\|^2\right).
\end{equation}
Note that for all  $x\in \HH_1$ we have $d_{N_0}(xa,x)=N_0(a)=1$ and similarly $d_{N_0}(xb,x)=1$. Moreover, for all $k\in \N$ and $x\in \HH_1$ we have $d_{N_0}(xc^k,x)=N_0(c^k)=N_0(0,-4k)=2\sqrt{k}$. Hence, using~\eqref{eq:assumption f} and the fact that the cardinality of
$B_m$ is bounded above and below by universal
 multiples of $m^4$, inequality~\eqref{eq:ANT} becomes 
 $D^2\gtrsim \sum_{k=1}^{m^2}\frac{N_0(c^k)^{2(1-\e)}}{k^2}\gtrsim \sum_{k=1}^{m^2} \frac{1}{k^{1+\e}}$. Letting $m$ tend to $\infty$ we deduce that $D^2\gtrsim \sum_{k=1}^{\infty} \frac{1}{k^{1+\e}}\gtrsim \frac{1}{\e}$, as
 required.\qed



  \bibliographystyle{abbrv}
  \bibliography{assouad}
\end{document}

  \begin{equation}\label{eq:edge assumption}
  d\left(u,u'\right)\le 4\left(\frac{\e}{32\log K}\right)^{\frac{i}{1-\e}},
  \end{equation}
  and recall that from the definition of M in~\eqref{eq:def M} we know that
  \begin{equation}\label{eq:M guarantee upper}
  d(u,v),d(u',v')\le 3\left(\frac{\e}{32\log K}\right)^{\frac{i-1}{1-\e}}.
  \end{equation}

  \begin{eqnarray}\label{eq:Holder}
\|F(x)-F(y)\|_2&\le& \nonumber\max_{k\in \{1,\ldots,N\}}\sum_{i=1}^{\infty}\left|f_i^k(x)-f_i^k(y)\right|\\\nonumber
 &\lesssim& \frac{(\log K)^{2(1-\e)}}{\e}\left(1+\frac{1}{\e\left(\log(1/\e)+\log\log K\right)}\right)d(x,y)^{1-\e}\\\nonumber
&\asymp& \left\{\begin{array}{ll}\frac{(\log K)^{2(1-\e)}}{\e}& \mathrm{if\ }\e\in \left(\frac{1}{\log\log K},\frac12\right]
\\\nonumber \frac{(\log K)^2}{\e^2\log\log K}& \mathrm{if\ }\e\in\left(\frac{1}{\log K},\frac{1}{\log\log K}\right]\\\nonumber
\frac{(\log K)^2}{\e^2\log(1/\e)}& \mathrm{if\ }\e\in \left(0,\frac{1}{\log K}\right]\end{array}\right\}\cdot d(x,y)^{1-\e} \\&\lesssim& \left(\frac{\kappa}{\e}\right)^{1+\theta} d(x,y)^{1-\e}.
\end{eqnarray}

\begin{multline}
A(i,j,k,y)\eqdef \left(\bigcap_{\ell\in J(i,y)\cap\{1,\ldots,j-1\}}\left\{B\left(x_\ell^i,R_{i\ell}^k\right)\cap B\left(y,\beta \tau^{\frac{i}{1-\e}}\right)=\emptyset\right\}\right)\\
\bigcap \left\{ B\left(x_j^i,R_{ij}^k\right)\cap B\left(y,\beta
\tau^{\frac{i}{1-\e}}\right)\notin\left\{\emptyset, B\left(y,\beta
\tau^{\frac{i}{1-\e}}\right)\right\}\right\}.
\end{multline}
\begin{equation}
Z(i,k,y)=\Omega\setminus \bigcup_{j\in J(i,y)} A(i,j,k,y).
\end{equation}

\newpage


Luukkainen and Movahedi-Lankarani~\cite{LM94}

For example, it was shown by Gupta-Lee-Krauthgamer~\cite{GKL03} that
a doubling metric tree admits an embedding into some $\R^N$; other
(very different) proofs of this result were obtained by
Lee-Naor-Peres~\cite{LNP09} and Gupta-Talwar~\cite{GT08}.